\documentclass[12pt,a4paper]{amsart}
\usepackage[DIV10, headinclude, BCOR 10mm]{typearea}
\usepackage{latexsym}
\usepackage{amssymb}
\usepackage{amscd}
\usepackage{graphicx}
\usepackage{paralist} 
\usepackage{array}
\usepackage{tabularx}
\usepackage{url}

\usepackage{color}
\usepackage{pdfcolmk}

\newtheorem{thm}{Theorem}[section]
\newtheorem{prop}[thm]{Proposition}
\newtheorem{lem}[thm]{Lemma}
\newtheorem{cor}[thm]{Corollary}

\theoremstyle{remark}
\newtheorem{rem}[thm]{Remark}
\newtheorem{ex}[thm]{Example}

\theoremstyle{definition}
\newtheorem{defn}[thm]{Definition}

\DeclareMathOperator{\Aut}{Aut}
\DeclareMathOperator{\Out}{Out}
\DeclareMathOperator{\Inn}{Inn}
\DeclareMathOperator{\PSL}{PSL}
\DeclareMathOperator{\SL}{SL}
\DeclareMathOperator{\GL}{GL}
\DeclareMathOperator{\Sp}{Sp}

\DeclareMathOperator{\Sol}{Sol}
\DeclareMathOperator{\IA}{IA}
\DeclareMathOperator{\RG}{RG}

\newcommand{\CC}{\mathcal{ C}}
\newcommand{\Cr}{C^*_{\text{\normalfont red}}}
\newcommand{\C}{\mathbb{ C}}
\newcommand{\N}{\mathbb{ N}}

\newcommand{\R}{\mathbb{ R}}
\newcommand{\Z}{\mathbb{ Z}}
\newcommand{\bN}{\N}

\newcommand{\Hyp}{\mathbb{ H}}
\newcommand{\Creg}{\mathcal C_{\text{\normalfont reg}}}
\newcommand{\Dreg}{\mathcal D_{\text{\normalfont reg}}}

\DeclareMathOperator{\CAT}{CAT}
\newcommand{\Catz}{\ensuremath{\CAT(0)}}

\makeatletter
\newcommand\norm{\bBigg@{0.8}}

\makeatother
\newcommand{\inparens}[2][flex]{\csname #1l\endcsname(#2%
  \csname #1r\endcsname)\mathclose{}}
\newcommand{\inangles}[2][flex]{\csname #1l\endcsname\langle#2%
  \csname #1r\endcsname\rangle\mathclose{}} 

\DeclareMathOperator{\im}{im}
  
\newcommand{\sv}[2][flex]{\csname #1l\endcsname\|#2%
  \csname #1r\endcsname\|} 
\newcommand{\lone}[2][flex]{\csname #1l\endcsname\|#2%
  \csname #1r\endcsname\|_1} 
\newcommand{\supn}[2][flex]{\csname #1l\endcsname\|#2%
  \csname #1r\endcsname\|_{\infty}} 

\newcommand{\ltb}[3][flex]{b^{(2)}_{#2}\inparens[#1]{#3}}

\title[Groups not presentable by products]
      {Groups not presentable by products}
\author{D.~Kotschick}
\address{Mathematisches Institut, {\smaller LMU} M\"unchen,
Theresienstr.~39, 80333~M\"unchen, Germany}
\email{dieter@member.ams.org}
\author{C.~L\"oh}
\address{Fakult\"at f\"ur Mathematik, Universit\"at Regensburg, 93040 Regensburg, Germany}
\email{clara.loeh@mathematik.uni-regensburg.de}
\date{\today; \copyright{\ D.~Kotschick and C.~L\"oh 2009--2011}}
\thanks{The first author is grateful to the Institute for Advanced Study in Princeton,
respectively the Mathematical Sciences Research Institute in Berkeley, for hospitality at
the beginning, respectively at the end, of the work on this paper. 
The second author is grateful for the support of the Hausdorff Research Institute for Mathematics in Bonn
during the {\smaller HIM} Trimester Program ``Rigidity.''}

  \usepackage{ifpdf}
  \ifpdf
    \usepackage[pdftex,colorlinks,linkcolor=blue,citecolor=blue,urlcolor=blue,a4paper,hypertexnames=true,plainpages=false]{hyperref}
    \hypersetup{pdfauthor={D.~Kotschick, C.~L\"oh},
                pdftitle ={Groups not presentable by products},
                }
    \usepackage[all]{xy}
    \SelectTips{cm}{}
  \else
    \usepackage[colorlinks,linkcolor=blue,citecolor=blue,urlcolor=blue,a4paper,hypertexnames=true,plainpages=false]{hyperref}
    \usepackage[all, dvips]{xy}
    \SelectTips{cm}{}
  \fi

\begin{document}

\begin{abstract}
  In this paper we study obstructions to presentability by products
  for finitely generated groups. Along the way we develop both the
  concept of acentral subgroups, and the relations between
  presentability by products on the one hand, and certain geometric
  and measure or orbit equivalence invariants of groups on the
  other. This leads to many new examples of groups not presentable by
  products, including all groups with infinitely many ends, the
  (outer) automorphism groups of free groups, Thompson's groups, and
  even some elementary amenable groups.
\end{abstract}

\maketitle


\section{Introduction}\label{s:intro}

In our previous paper~\cite{KL} we introduced a class of infinite
groups, called groups not presentable by products.  Our motivation 
was that certain closed manifolds whose fundamental groups belong
to this class turned out to have special properties; in particular
some such manifolds are not dominated by non-trivial product manifolds.  The
purpose of the present paper is to discuss groups not presentable by products
more systematically, and, in particular, to provide further examples of such groups, 
going far beyond the examples given in~\cite{KL}. 
First, we recall the definition.
 \begin{defn}\label{repgroupsdef}(\cite{KL})
    An infinite group $\Gamma$ is \emph{not presentable by a product} if, for
    every homomorphism $\varphi \colon \Gamma_1 \times\Gamma_2
    \longrightarrow\Gamma$ onto a subgroup of finite index, one of the
    factors $\Gamma_i$ has finite image $\varphi
    (\Gamma_i)\subset\Gamma$.
  \end{defn}
For the fundamental groups of closed Riemannian manifolds of strictly negative sectional curvature this property holds,
essentially by the proof of Preissmann's theorem. Generalizing this observation, we previously proved:
\begin{thm}\label{t:alg}{\rm (\cite[Theorem~1.5]{KL})}
  The following groups are not presentable by products:
\begin{enumerate}[\normalfont{(MCG)}]
\item[\normalfont{(H)}] hyperbolic groups that are not virtually
  cyclic,
\item[\normalfont{(N-P)}] fundamental groups of closed Riemannian
  manifolds of non-positive sectional curvature of rank one and of
  dimension~$\geq 2$,
\item[\normalfont{(MCG)}] mapping class groups of closed oriented
  surfaces of genus $\geq 1$.
\end{enumerate}
\end{thm}
As a consequence of the discussion in this paper, we extend Theorem~\ref{t:alg} in several directions. 

After the preliminary  Section~\ref{s:prelim}, this paper consists of two main parts. In the first part, 
comprising Sections~\ref{s:acentral}--\ref{s:bc}, we develop criteria to 
show that groups are not presentable by products. Although numerous examples are interspersed 
in this first part, we then devote the second part of the paper, comprising Sections~\ref{s:thompson}, \ref{s:aut}
and \ref{s:free}, to systematically testing the criteria from the first part on certain interesting classes of 
groups, leading to further examples.

In Section~\ref{s:acentral} we discuss groups with acentral subgroups, a notion tailored to the 
analysis of centralisers \`a la Preissmann. This discussion subsumes most of the ad hoc arguments
that went into the proof of Theorem~\ref{t:alg}, but it also applies to other interesting examples, 
such as free products and certain elementary amenable groups obtained as semidirect products.

In Section~\ref{s:cost} we develop obstructions to presentability by products coming from $L^2$-Betti
numbers, from cost in the sense of Levitt and Gaboriau, and from the rank gradient in the sense of Lackenby.

In Section~\ref{s:Powers} we consider the Powers property introduced by de~la~Harpe, and in Section~\ref{s:bc}
the second bounded cohomology with coefficients in the regular representation, as pioneered by Burger, Monod and Shalom. These considerations 
show that the hyperbolicity in statement~(H) of Theorem~\ref{t:alg} can be replaced by a ``cohomological negative
curvature'' assumption. Needless to say, these ``negative curvature'' obstructions do not apply to amenable groups,
although some of them are not presentable by products and are amenable (sic!) to a direct analysis of centralisers
\`a la Preissmann.

In Section~\ref{s:thompson} we test our criteria on Richard Thompson's groups, which are not elementary 
amenable, but could still be amenable.  In Section~\ref{s:aut} we discuss the automorphism groups of free groups,
proving the natural generalization to this class of the statement about mapping class groups in Theorem~\ref{t:alg}.
Finally in Section~\ref{s:free} we discuss groups with infinitely many ends, in particular free products and their applications to
connected sum decompositions of manifolds dominated by products.

The final Section~\ref{s:final} contains some further extensions of our results, and a discussion of the context
in geometric and measurable group theory. The Appendix summarizes the applicability of different criteria 
to various classes of groups.

\subsection*{Acknowledgement}
This paper was begun in 2008 as we were trying to understand some comments that N.~Monod kindly made on~\cite{KL}.
More recently, he has again offered some crucial insights. We are very grateful to N.~Monod for his generous
contributions to this work. In particular, most of the results and proofs in Sections~\ref{s:L2Betti}, \ref{ss:cost}, 
\ref{s:bc} and~\ref{ss:normal} were suggested by him.

\section{Preliminaries}\label{s:prelim}

Throughout this paper we only consider finitely generated groups. This restriction is necessary for some of the results we use,
but represents no significant loss of generality. In particular, it is always satisfied when considering fundamental
groups of compact manifolds, as in the context of the domination relation for manifolds~\cite{KL}.

We recall a few of the elementary properties of groups not presentable
by products developed in~\cite{KL}. We do not give any proofs in this section as 
all the results that need proof were proved in Subsection~3.1 of loc.~cit.

Consider a homomorphism
$\varphi\colon\Gamma_1\times\Gamma_2\longrightarrow\Gamma$ of groups.
Without loss of generality we can replace each $\Gamma_i$ by its image
in $\Gamma$ under the restriction of~$\varphi$, so that we may assume
the factors $\Gamma_i$ to be subgroups of $\Gamma$ and $\varphi$ to be
multiplication in~$\Gamma$.

\begin{lem}\label{l:C}
  If a group is not presentable by a product, then every finite index
  subgroup has finite centre.
\end{lem}

The following is a kind of converse to this observation:

\begin{prop}\label{p:C}
  If every subgroup of finite index in a group~$\Gamma$ has trivial
  centre, then $\Gamma$ is irreducible if and only if it is not
  presentable by a product.
\end{prop}

The proof of this proposition is based on the following lemma.

\begin{lem}\label{l:center}
  Let $\Gamma_1$, $\Gamma_2\subset\Gamma$ be commuting subgroups of a
  group~$\Gamma$ with the property that $\Gamma_1 \cup \Gamma_2$
  generates~$\Gamma$. Then the multiplication
  homomorphism~$\varphi\colon\Gamma_1\times\Gamma_2\longrightarrow\Gamma$
  is well-defined and surjective and the following statements hold:
\begin{enumerate}
\item the intersection~$\Gamma_1\cap \Gamma_2\subset\Gamma$ is a
  subgroup of the centre of~$\Gamma$, and
\item the kernel of~$\varphi$ is isomorphic to the Abelian
  group~$\Gamma_1\cap \Gamma_2$.
\end{enumerate}
\end{lem}

This gives the following exact sequences
\begin{equation} 
    1 \longrightarrow \Gamma_1 \cap \Gamma_2
      \longrightarrow \Gamma
      \longrightarrow \Gamma/(\Gamma_1 \cap \Gamma_2)
      \longrightarrow 1 \label{eq:2}\ ,
\end{equation}
\begin{equation}
    1 \longrightarrow \Gamma_1 \cap \Gamma_2 
       \longrightarrow \Gamma_1 \times \Gamma_2
       \longrightarrow \Gamma
       \longrightarrow 1\label{eq:1} \ .
  \end{equation}  

Sometimes it is convenient to replace a given group by a subgroup of
finite index.  This transition does not affect presentability by
products by the following straightforward observation:

\begin{lem}\label{l:finindexpres}
  Let $\Gamma$ be a group. A finite index subgroup of~$\Gamma$ is
  presentable by a product if and only if $\Gamma$ is.
\end{lem}

\section{Acentral subgroups and acentral extensions}\label{s:acentral}

In this section we define groups with acentral subgroups and acentral extensions and prove that they are not presentable by products. 
We shall give various examples, including in particular elementary amenable
groups that are not presentable by products because they are extensions of this type.

\begin{defn}
  Let $\Gamma$ be a group. A subgroup~$A$ of~$\Gamma$ is called
  \emph{acentral} if for every~$g \in A \setminus \{1\} \subset
  \Gamma$ the centraliser~$C_\Gamma(g)$ is contained in~$A$.  

  An extension~$1 \longrightarrow N \longrightarrow \Gamma
  \longrightarrow Q \longrightarrow 1$ of groups is \emph{acentral} 
  if the normal subgroup~$N$ is acentral.
\end{defn}

Our interest in these notions stems from the following result.
\begin{prop}\label{prop:acentral}
  Groups containing infinite acentral subgroups of infinite index
  are not presentable by products.
\end{prop}
\begin{proof}
 The proof, in the same spirit as the ad hoc arguments of~\cite{KL}, 
  consists of a careful analysis of the commutation relations in a group~$\Gamma$
  containing an infinite acentral subgroup~$A$ of infinite index.

  Assume for a contradiction that $\Gamma$ is presentable by a
  product. Then there are commuting infinite subgroups~$\Gamma_1$ and $\Gamma_2$
  such that the multiplication homomorphism~$\varphi \colon \Gamma_1
  \times \Gamma_2 \longrightarrow \Gamma$ is well-defined and has
  finite index image in~$\Gamma$.

  As a first step we show that $\Gamma_i \cap A = 1$. If $\Gamma_1 \cap A$ 
  contained a non-trivial element~$g$, then -- because $A$ is acentral and the
  $\Gamma_i$ commute -- we would obtain~$\Gamma_2
  \subset C_\Gamma(g) \subset A$. Applying acentrality again, we then
  deduce that also~$\Gamma_1 \subset A$. But then $\im \varphi \subset
  A$, contradicting the fact that $\im\varphi$ has finite index
  in~$\Gamma$ and $A$ has infinite index in $\Gamma$. Therefore, indeed
  $\Gamma_1 \cap A = 1$ and $\Gamma_2 \cap A = 1$.

  As a second step we show that even $(\im \varphi )\cap A = 1$. Assume for
  a contradiction that $(\im \varphi )\cap A \neq 1$ and let~$g \in (\im
  \varphi )\cap A \setminus \{1\}$. Because $\Gamma_1$ and $\Gamma_2$
  commute and because $\Gamma_1 \cup \Gamma_2$ generates~$\im
  \varphi$, we find elements~$g_1 \in \Gamma_1$ and $g_2 \in \Gamma_2$
  such that~$g = g_1 \cdot g_2$; notice that $g_1 \in C_\Gamma(g)$.
  Therefore, acentrality implies that~$g_1 \in A$, and hence -- by the
  first step -- we have $g_1 = 1$. Applying the first step
  to~$\Gamma_2$, we obtain also $g_2 =1$, which contradicts~$g \neq
  1$.  So $(\im \varphi )\cap A = 1$.

  As the third and last step we show that any subgroup~$\Gamma'$ of~$\Gamma$ with
  $\Gamma' \cap A = 1$ cannot have finite index in~$\Gamma$.
 Since $A$ is infinite, if $\Gamma'$ had only finitely many cosets in $\Gamma$, then by the 
pigeonhole principle there
would be a coset, say $g\Gamma'$, containing infinitely many elements of $A$.
In particular $g\Gamma'\cap A$ would contain two elements $a_1\neq a_2$. But then
it would follow that $1\neq a_2^{-1}\cdot a_1\in \Gamma'\cap A$, which would be a contradiction.

Combining the second and third steps we reach the conclusion that $\im\varphi$ 
  cannot have finite index in~$\Gamma$. Thus, $\Gamma$ is not
  presentable by a product after all. 
\end{proof}

\begin{cor}
\label{cor:acentralext}
  Let $1 \longrightarrow N \longrightarrow \Gamma \longrightarrow Q
  \longrightarrow 1$ be an acentral extension of groups with $N$ and
  $Q$ infinite. Then $\Gamma$ is not presentable by a product.
\end{cor}

Before proceeding, we would like to point out that many instances of the ad hoc arguments
of~\cite{KL} can be subsumed under the result of Proposition~\ref{prop:acentral}. For 
example, in the mapping class groups of surfaces of genus $\geq 3$ there are pseudo-Anosov
elements $g$ with the property that the cyclic subgroup generated by $g$ is acentral. Thus
the case (MCG) in Theorem~\ref{t:alg} follows from Proposition~\ref{prop:acentral}. Note that
it is crucial here that we consider acentral subgroups that are not necessarily normal. This is 
also true in the next example.

\begin{ex}\label{ex:free}
Let $\Gamma=\Delta_1\star\Delta_2$ be a non-trivial free product, and $g\in\Gamma$ an 
element that is not contained in a conjugate of one of the free factors, e.g., $g=g_1g_2g_1^{-1}g_2^{-1}$
with $g_i\in\Delta_i\setminus\{ e\}$. Then the centraliser $C_{\Gamma}(g)$ is an infinite acentral 
subgroup of $\Gamma$; see~\cite[Problem~28 on p.~196]{MKS}. 
As soon as one of the $\Delta_i$ has order $>2$, the index of $C_{\Gamma}(g)$
in $\Gamma$ is also infinite, so that $\Gamma$ is not presentable by a product.
\end{ex}


\begin{cor}
  \label{cor:semidirect}
  Let $\Gamma$ denote the semi-direct product group~$N \rtimes_\alpha Q$,
  where $N$ is a non-trivial Abelian group, $Q$ is an infinite group and the
  action on~$N$ given by~$\alpha \colon Q \longrightarrow \Aut(N)$ is free
  outside~$0 \in N$. 
    Then the extension~$0 \longrightarrow N \longrightarrow
      \Gamma \longrightarrow Q \longrightarrow 1$ is acentral and $N$
      is infinite.
     In particular, $\Gamma$ is not presentable by a product.
\end{cor}
\begin{proof}
  In view of Corollary~\ref{cor:acentralext}, it suffices to prove that $N$ is infinite and
  acentral.  Infiniteness is clear since the infinite group~$Q$ acts freely on the (non-empty) set~$N
  \setminus \{0\}$.

  Let $g \in N \setminus\{0\}$ and let $g' \in C_\Gamma(g)$; we write $g
  = (n,1)$ and $g' = (n', q')$ with~$n$,~$n' \in N$ and $q' \in Q$. By
  definition of the semi-direct product, we obtain
  \begin{align*}
    g  \cdot g' & = \bigl( n + n', q' \bigr), \\
    g' \cdot g  & = \bigl( n'+ \alpha(q')(n), q' \bigr), 
  \end{align*}
  and hence $n + n' = n' + \alpha(q')(n)$; because $N$ is Abelian, we
  use ``+'' for the group structure in~$N$. In particular, $n =
  \alpha(q')(n)$. As $\alpha$ acts freely on~$N$ outside~$0$ and
  $n\neq 0$, this implies~$q' = 1$, i.e., $g' \in N \rtimes_\alpha 1 =
  N$. Thus $N$ is indeed acentral.
\end{proof}

This corollary provides us with explicit examples of elementary amenable
groups that are not presentable by products, by taking semi-direct
products of infinite amenable groups~$Q$ with suitable actions on
Abelian groups~$N$. Note that the obstructions to presentability by products
developed below
coming from rank gradient, cost and $L^2$-Betti numbers (Section~\ref{s:cost}),
from the Powers property (Section~\ref{s:Powers}), or from bounded 
cohomology (Section~\ref{s:bc}), vanish for amenable groups.

\begin{ex}
    \label{ex:amenable}
    Let $\alpha \colon \Z \longrightarrow \SL(2,\Z)$ be given by the
    matrix    
    \[ A = \begin{pmatrix}
              2 & 1 \\
                        1 & 1
           \end{pmatrix} \ .
    \]
    Then the action of~$\Z$ on~$\Z^2$ given by~$A$ is free
    outside~$0$ and Corollary~\ref{cor:semidirect} shows that the
    corresponding semi-direct product~$\Gamma = \Z^2 \rtimes_\alpha \Z$ is not
    presentable by a product. 
    On the other hand, $\Gamma$ is solvable and thus amenable.
\end{ex}

    In this example $\Gamma$ is the fundamental group of a
    $T^2$-bundle $M$ over $S^1$ whose monodromy is the Anosov
     diffeomorphism given by $A$ acting linearly on $\R^2/\Z^2$.
    This torus bundle carries the solvable Thurston geometry~$\Sol^3$,
    and $\Gamma$ is a lattice in~$\Sol^3$. 
    For all such lattices we have:
\begin{cor}\label{c:Sol}
Let $\Gamma$ be any cocompact lattice in~$\Sol^3$. Then $\Gamma$ is
not presentable by a product.
\end{cor}
\begin{proof}
Any such lattice has a finite index subgroup that is the fundamental group
of the mapping torus of a hyperbolic torus diffeomorphism; see~\cite[Theorem~5.3(i)]{Scott}.
The discussion in Example~\ref{ex:amenable} applies to this finite index subgroup.
\end{proof}

The previous example can be generalized to higher dimensions.
\begin{ex}
    Let $n \in \bN_{\geq 2}$ and let $\alpha \colon \Z \longrightarrow
    \GL(n,\Z)$ be given by a matrix~$A \in \GL(n,\Z)$ such that no
    non-trivial power of~$A$ has eigenvalue~$1$. Then the action of~$\Z$ on~$\Z^n$ given by~$A$ is free
    outside~$0$ and Corollary~\ref{cor:semidirect} shows that the
    corresponding semi-direct product~$\Z^n \rtimes_\alpha \Z$ is not
    presentable by a product. Again the group is solvable
    and thus amenable.
\end{ex}


\section{$L^2$-Betti numbers, cost, and rank gradient}\label{s:cost}

In this section we discuss obstructions to presentability by products
coming from $L^2$-Betti numbers, from cost, and from the rank
gradient. 

\subsection{$L^2$-Betti numbers}\label{s:L2Betti}

Like the ordinary Betti numbers, \mbox{$L^2$-Bet}\-ti numbers of
groups can be defined as dimensions of certain homology modules,
namely as the von~Neumann dimensions of homology with coefficients in the group von~Neumann
algebra. For a thorough treatment of $L^2$-invariants we refer
the reader to L\"uck's book~\cite{lueckl2}.

The first $L^2$-Betti number is an obstruction for groups 
to be presentable by products: 

\begin{prop}
\label{firstl2prop}
  If the group~$\Gamma$ is presentable by a product then $\ltb 1 \Gamma = 0$.
\end{prop}
\begin{proof}
  Assuming that $\Gamma$ is presentable by a product we find two
  infinite commuting subgroups~$\Gamma_1$ and~$\Gamma_2$ of~$\Gamma$
  with the property that the multiplication homomorphism~$\varphi
  \colon \Gamma_1\times\Gamma_2 \longrightarrow \Gamma$ is surjective
  onto a finite index subgroup~$\Gamma' := \im \varphi$ in~$\Gamma$.

  As $L^2$-Betti numbers are multiplicative with respect to
  finite index subgroups~\cite[Theorem~6.54(6)]{lueckl2}, we have
  \[ \ltb 1 \Gamma = \frac 1 {[\Gamma : \Gamma']} \cdot \ltb 1 {\Gamma'}.
  \]
  In particular, it suffices to prove that~$\ltb 1 {\Gamma'} = 0$.

 We now divide the discussion into two cases:
  \begin{enumerate}
  \item \emph{The group~$\Gamma_1 \cap \Gamma_2$ is infinite.}
    Since~$\Gamma_1 \cap \Gamma_2$ is Abelian, it is amenable.  Thus,
    by the exact sequence~\eqref{eq:2}, the group~$\Gamma'$ has an
    infinite amenable normal subgroup, which implies that~$\ltb 1
    {\Gamma'} = 0$; see~\cite[Theorem~7.2 (1) and (2)]{lueckl2}.
  \item \emph{The group~$\Gamma_1 \cap \Gamma_2$ is finite.} In this
    case the exact sequence~\eqref{eq:1} implies that
      \[ \ltb 1 {\Gamma'} =     |\Gamma_1\cap \Gamma_2| 
                          \cdot \ltb 1 {\Gamma_1 \times \Gamma_2} \ ;
      \]
      compare~\cite[Exercise~7.7 and p.~534f]{lueckl2}.
      Moreover, $\ltb 1 {\Gamma_1 \times \Gamma_2}$ can be computed by
      the K\"unneth formula~\cite[Theorem~6.54(5),(8)]{lueckl2}
      \begin{align*}
            \ltb 1 {\Gamma_1 \times \Gamma_2}
        & = \ltb 0 {\Gamma_1} \cdot \ltb 1 {\Gamma_2}
          + \ltb 1 {\Gamma_1} \cdot \ltb 0 {\Gamma_2} \\
        & = \frac 1 {|\Gamma_1|} \cdot \ltb 1 {\Gamma_2}
          + \ltb 1 {\Gamma_1} \cdot \frac 1 {|\Gamma_2|} \\
        & = 0 \ ;
      \end{align*}
      where the last equality holds because the groups~$\Gamma_1$
      and~$\Gamma_2$ are infinite. Thus, it follows that~$\ltb 1
      {\Gamma'} = 0$.
  \end{enumerate}

  Hence, we obtain~$\ltb 1 \Gamma = 1/[\Gamma : \Gamma'] \cdot \ltb 1
  {\Gamma'} = 0$, as desired.
\end{proof}

\begin{rem}
  The vanishing result in Proposition~\ref{firstl2prop} does not
  extend to the higher $L^2$-Betti numbers.  However, the proof of
  case~(1) does extend. Therefore, we do get restrictions on the
  higher $L^2$-Betti numbers of groups that are presentable by
  products. For example, if $\Gamma$ is presentable by a product and
  $\ltb 2 {\Gamma}\neq 0$, then $\Gamma$ is a quotient of a direct
  product $\Gamma_1\times\Gamma_2$ by a finite central subgroup, and
  both $\ltb 1 {\Gamma_1}\neq 0$ and $\ltb 1 {\Gamma_2} \neq 0$.  
  This follows from the K\"unneth formula as in case~(2) of the proof
  above. 
\end{rem}

\subsection{Expensive groups}\label{ss:cost}

The concept of cost was introduced by Levitt and developed extensively
by Gaboriau~\cite{gaboriauInvent}.  It is a dynamical/ergodic
invariant of discrete groups.  We shall use the lecture notes of
Kechris and Miller~\cite{KM} as our reference for this concept and for
the properties we need.

The cost $\CC(\Gamma)$ of a countable group~$\Gamma$ is either
infinite or a non-negative real number. For finite groups one has
$\CC(\Gamma)= 1 - 1/\vert\Gamma\vert$, and for infinite groups one
has~$\CC(\Gamma)\geq 1$.

\begin{defn}[\cite{KM}] An infinite countable group~$\Gamma$ is {\it cheap} if
  $\CC(\Gamma)=1$; it is {\it expensive} if $\CC(\Gamma)>1$.
\end{defn}

If $\Gamma'\subset\Gamma$ is a subgroup of finite index, then
$\CC(\Gamma')-1 = [\Gamma\colon\Gamma']\cdot
(\CC(\Gamma)-1)$~\cite[Th\'eor\`eme~VI.1]{gaboriauInvent}\cite[Theorem~34.1]{KM}. Therefore,
the property of being cheap, or expensive, is invariant under passage
to finite index sub- or supergroups. Similarly, if $\Gamma'$ is a
finite normal subgroup of~$\Gamma$, then $\Gamma/\Gamma'$ is cheap if
and only if $\Gamma$ is
cheap~\cite[Th\'eor\`eme~VI.19]{gaboriauInvent}.

\begin{prop}
Expensive groups are not presentable by products.
\end{prop}
\begin{proof}
  Suppose that $\Gamma$ is expensive and presentable by a
  product. Then, as the property of being expensive is invariant under
  passage to finite index subgroups, Lemma~\ref{l:finindexpres} allows
  us to assume that we have commuting subgroups
  $\Gamma_1,\Gamma_2\subset\Gamma$ such that the multiplication
  $\varphi\colon\Gamma_1\times\Gamma_2\longrightarrow\Gamma$ is
  surjective.  If $\Gamma_1\cap\Gamma_2$ is infinite, then $\Gamma$
  has infinite centre, and so $\Gamma$ is
  cheap~\cite[VI.26~(a)]{gaboriauInvent}\cite[Corollary~35.3]{KM}. If
  $\Gamma_1\cap\Gamma_2$ is finite, then in view of the exact
  sequence~(\ref{eq:1}) and the fact that the property of being cheap
  is invariant under passage to quotients by finite normal subgroups,
  it suffices to check that $\Gamma_1\times\Gamma_2$ is cheap. This
  last assertion is known to be true as soon as both $\Gamma_i$ are
  infinite~\cite[Proposition~VI.23]{gaboriauInvent}\cite[Theorem~33.3]{KM}.
\end{proof}

\subsection{Rank gradient}\label{s:RG}

The rank gradient was introduced by Lackenby~\cite{lack}, and has
recently been further developed by Ab\'ert and Nikolov~\cite{AN}.

For any finitely generated group $\Gamma$, let $d(\Gamma )$ be the
minimal number of generators.  Then the \emph{rank gradient} is defined to be
\[
\RG(\Gamma) = \inf_{\Gamma'\subset \Gamma}\frac{d(\Gamma')-1}{[\Gamma\colon\Gamma']} \ ,
\]
with the infimum taken over all finite index subgroups
$\Gamma'\subset\Gamma$. (This is sometimes called the absolute rank
gradient. Often only normal subgroups are considered, but this makes no
difference.)  Of course, if $\Gamma$ has few subgroups of finite index,
this definition may not be very meaningful.  In the extreme case when
$\Gamma$ has no subgroups of finite index at all, one clearly has
$\RG(\Gamma) = d(\Gamma)-1$. This explains why results about the rank
gradient often involve assumptions that ensure the existence of
sufficiently many finite index subgroups.

The basic properties of the rank gradient immediately give the following:
\begin{prop}\label{p:RG}
  If a residually finite group $\Gamma$ is presentable by a product,
  then $\RG(\Gamma)=0$.
\end{prop}
\begin{proof}
  Suppose $\Gamma$ is presentable by a product. Then there are
  infinite commuting subgroups $\Gamma_1, \Gamma_2\subset\Gamma$ such
  that the multiplication map
  $\varphi\colon\Gamma_1\times\Gamma_2\longrightarrow\Gamma$ is
  surjective onto a finite index subgroup~$\Gamma'\subset\Gamma$. It
  suffices to prove~$\RG(\Gamma')=0$.

  If $\Gamma_1\cap\Gamma_2$ is infinite, then $\Gamma'$ has infinite
  centre, and so its rank gradient vanishes~\cite[Theorem~3]{AN}. 

  If $\Gamma_1\cap\Gamma_2$ is finite, we argue as follows.  By
  assumption, both $\Gamma_i$ are infinite. As they are subgroups of a
  residually finite group, they are themselves residually finite. As
  the two groups commute, they are both normal in~$\Gamma'$, and we
  have the two exact sequences
  \begin{align*}
    1\longrightarrow \Gamma_1
     \longrightarrow \Gamma' 
     \longrightarrow \Gamma_2/(\Gamma_1\cap\Gamma_2) 
     \longrightarrow 1 \\
    1\longrightarrow \Gamma_2
     \longrightarrow \Gamma' 
     \longrightarrow \Gamma_1/(\Gamma_1\cap\Gamma_2) 
     \longrightarrow 1 \text{\makebox[0pt][l]{\ ;}}
  \end{align*}
  the epimorphisms are given by composing the isomorphism~$\Gamma'
  \cong \Gamma_1\times\Gamma_2/(\Gamma_1 \cap \Gamma_2)$ with the
  homomorphisms induced by the projection
  from~$\Gamma_1\times\Gamma_2$ onto its factors.  The lower sequence
  shows that $\Gamma_1/(\Gamma_1\cap\Gamma_2)$ is finitely generated,
  and since $\Gamma_1\cap\Gamma_2$ is finite, we conclude that
  $\Gamma_1$ is itself finitely generated (recall that $\Gamma$ (and
  hence~$\Gamma'$) is finitely generated by our standing
  convention). Now we can apply a result of Ab\'ert and
  Nikolov~\cite[Prop.~13]{AN} to the first extension above to conclude the
  vanishing of $\RG(\Gamma')$. The subgroup $\Gamma_1$ is finitely
  generated, and the quotient $\Gamma_2/(\Gamma_1\cap\Gamma_2)$ has
  subgroups of arbitrarily large index since $\Gamma_2$ is infinite
  and residually finite, and $\Gamma_1\cap\Gamma_2$ is finite.
  This completes the proof.
\end{proof}

\begin{ex}
  Let $\Gamma$ be a finitely generated infinite simple group. By classical 
  work of Higman and Thompson, such groups exist, and may even be chosen to be finitely
  presentable. Then $\Gamma\times\Gamma$ is
  presentable by a product and has positive rank gradient since
  it has no proper subgroups of finite index.

  Note that $\Gamma$ itself is not presentable by a product since it has no
  non-trivial normal subgroups.
\end{ex}

\subsection{The relation between the first $L^2$-Betti number, cost, and the rank gradient}\label{ss:comparison}

There is a remarkable connection between cost and the first $L^2$-Betti number, which
shows that the obstruction to presentability by a product coming from the first $L^2$-Betti number is a special case 
of the obstruction provided by the cost. 
\begin{thm}[Gaboriau~\protect{\cite[Corollaire~3.23]{gaboriauIHES}}]\label{t:gab}
Every infinite group $\Gamma$ satisfies~$\CC(\Gamma)-1\geq \ltb 1 {\Gamma}$. In particular, groups with 
positive first $L^2$-Betti numbers are expensive.
\end{thm}
It is unknown whether this inequality is ever strict.
For residually finite groups one also has:
\begin{thm}[Ab\'ert and Nikolov~\protect{\cite[Theorem~1]{AN}}]\label{t:AN}
If $\Gamma$ is residually finite, then $\RG(\Gamma)\geq\CC(\Gamma)-1$, with equality if $\Gamma$ has 
fixed price.
\end{thm}
We refer the reader to the papers by Ab\'ert and Nikolov~\cite{AN} and by Osin~\cite{osin} for further discussions
of these results and their relations to open problems and conjectures in group theory and in three-dimensional 
topology.

For residually finite groups, the positivity of the rank gradient is
the strongest one of the three obstructions to presentability by
products discussed in this section.  A large class of non-examples for
this obstruction comes from the following observation, generalising a
vanishing result for $L^2$-Betti numbers~\cite[Theorem~1.39]{lueckl2}.

\begin{lem}[Lackenby~\protect{\cite[p.~365/366]{lack}}]\label{l:lack}
  The rank gradient vanishes for fundamental groups of mapping tori.
\end{lem}

\begin{ex}\label{ex:hyper}
By Thurston's theorem, the mapping torus of a pseudo-Anosov diffeomorphism 
of a surface of genus $\geq 2$ is a hyperbolic three-manifold. Its
  fundamental group is residually finite, therefore the vanishing of
  its rank gradient given by Lemma~\ref{l:lack} implies that this
  hyperbolic group is cheap.
\end{ex}

Another non-example is the following:
\begin{prop}\label{p:MCG}
Mapping class groups of surfaces of genus $\geq 3$ have vanishing rank gradient.
\end{prop}
Since mapping class groups are residually finite by a result of Grossman~\cite{gross}, this Proposition
implies, via Theorem~\ref{t:AN}, that mapping class groups are cheap. This last assertion was previously known
by a recent result due to Kida~\cite{Kida}.
\begin{proof}
Let $\Gamma_g$ be the mapping class group of a closed oriented surface of genus~$g\geq 3$.
We can apply a result of Ab\'ert and Nikolov~\cite[Proposition~13]{AN} to the extension 
\[
1\longrightarrow \mathcal{I}_g \longrightarrow \Gamma_g\longrightarrow \Sp(2g;\Z) \longrightarrow 1 
\]
to conclude $\RG(\Gamma_g)=0$. Here $\mathcal{I}_g$ is the Torelli
group, which is finitely generated for genus $\geq 3$ by a result of
Johnson~\cite{Johnson}. Clearly the symplectic group has the required
property to admit finite quotients of arbitrarily large order.
\end{proof}

\section{The Powers property}\label{s:Powers}

For a countable group $\Gamma$, let $\Cr(\Gamma)$ denote its reduced $C^*$-algebra.
One says that the group is \emph{$C^*$-simple} if $\Cr(\Gamma)$ has no proper two-sided ideals.
It is a now classical result of Powers that the free group on two generators is $C^*$-simple.
De~la~Harpe~\cite{Harpe1} extracted from Powers's proof a combinatorial property of groups
that ensures $C^*$-simplicity. He calls this property the \emph{Powers property}, and calls groups that
have the property Powers groups. We refer to his recent survey~\cite{Harpe2} for the definitions
and an extensive bibliography of results on the class $\mathcal{P}$ of Powers groups.

Basic results about Powers groups mentioned in~\cite{Harpe2} imply the following.
\begin{prop}\label{p:Powers}
A group with the Powers property is not presentable by products.
\end{prop}
\begin{proof}
First of all, if $\Gamma$ is a Powers group, so is every finite index subgroup~\cite[Proposition~1(c)]{Harpe1}.
Therefore, if we have subgroups $\Gamma_1$, $\Gamma_2\subset\Gamma$ for which the multiplication
is surjective onto a finite index subgroup, we just pass to this subgroup.
Now a $C^*$-simple group does not contain any amenable normal subgroup, in particular it has 
trivial centre. Therefore, by the discussion in Section~\ref{s:prelim}, we conclude that $\Gamma_1\times\Gamma_2$
is a Powers group. But this contradicts a result of Promislow~\cite{Pr}; compare also~\cite[Proposition~14(i)]{Harpe2}.
\end{proof}
This result is true for Powers groups only. The less restrictive property of $C^*$-simplicity is preserved under taking direct products, 
and so cannot be an obstruction against presentability by products. The same remark applies to the ``weak Powers
property'' discussed in~\cite{Harpe2}.

The fact that $C^*$-simple groups, and therefore Powers groups, have trivial centre, implies that many
standard examples cannot be Powers groups.
\begin{ex}\label{ex:SL2}
The centre of $\SL(2,\Z)$ has order $2$. Therefore, this is not a Powers group. It follows that hyperbolic
groups, or groups with infinitely many ends, are not always Powers groups.
\end{ex}

\begin{ex}\label{ex:MCG2}
The mapping class group of a closed genus $2$ surface also has centre of order $2$, generated by the hyperelliptic involution. 
It follows that this mapping class group is not a Powers groups. Similarly, in higher genus the hyperelliptic mapping class group
is not a Powers group.
\end{ex}

Nevertheless, the class $\mathcal{P}$ of Powers groups contains, among others, the following groups:
\begin{enumerate}
\item torsion-free hyperbolic groups that are not virtually cyclic (de~la~Harpe \cite{Harpe1,Harpe2}),
\item free products $\Delta_1\star\Delta_2$ with $\vert\Delta_i\vert>i$ (de~la~Harpe~\cite[Proposition~8]{Harpe1}),
\item mapping class groups of surfaces of genus at least~$3$ (Bridson--de~la~Harpe \cite[Theorem~2.2]{BridsonHarpe}).
\end{enumerate}

\section{Bounded cohomology}\label{s:bc}

Monod and Shalom~\cite{MSjdg,MSannals} introduced and studied the
following class of groups; compare also the paper by Mineyev, Monod
and Shalom~\cite{MMS}. (A detailed treatment of bounded
cohomology~$H^*_b$ of groups is given in Monod's book~\cite{monod}).

\begin{defn}[\cite{MSannals}]
  A countable group $\Gamma$ is in $\Creg$ if $H^2_b(\Gamma;\ell^2(\Gamma))\neq 0$.
\end{defn}
The class $\Creg$ contains, among others, the following groups:
\begin{enumerate}
\item hyperbolic groups that are not virtually cyclic (Mineyev--Monod--Sha\-lom~\cite[Theorem~3]{MMS}; see also~\cite{MSjdg,Ham}),
\item groups with infinitely many ends (Monod--Shalom~\cite[Corollary~7.9]{MSjdg}),
\item mapping class groups of surfaces of genus at least~$2$
  (Hamenst\"adt~\cite[Theorem~4.5]{Ham}).
\end{enumerate}
The results of Hamenst\"adt~\cite{Ham} hold more generally for all groups acting by isometries on a Gromov-hyperbolic
space, as long as the action satisfies a so-called weak acylindricity property.

\begin{prop}\label{p:Creg}
Groups in the class~$\Creg$ are not presentable by products.
\end{prop}
\begin{proof}
  This is implicit in the work of Monod and Shalom~\cite[Section~7]{MSannals}.

  Assume for a contradiction that $\Gamma$ is a group in~$\Creg$ that
  is presentable by a product.  If a group is in~$\Creg$,
  then so are all its finite index
  subgroups~$\Gamma$~\cite[Lemma~7.5]{MSannals}. Therefore, by
  Lemma~\ref{l:finindexpres}, we may
  assume that $\Gamma$ contains commuting subgroups~$\Gamma_1$ and
  $\Gamma_2$ such that the multiplication homomorphism
  $\Gamma_1\times\Gamma_2\longrightarrow\Gamma$ is
  surjective. Now if $\Gamma_1\cap\Gamma_2$ is infinite, then $\Gamma$
  contains an infinite amenable normal subgroup by the exact
  sequence~\eqref{eq:2}, which contradicts the assumption that
  $\Gamma$ is in $\Creg$~\cite[Prop.~7.10~(ii)]{MSannals}. If
  $\Gamma_1\cap\Gamma_2$ is finite, then the exact
  sequence~\eqref{eq:1} and~\cite[Lemma~7.3]{MSannals} imply that
  $\Gamma_1\times\Gamma_2$ is in $\Creg$. If both $\Gamma_i$ are
  infinite, this is impossible~\cite[Prop.~7.10~(iii)]{MSannals}.
\end{proof}

Proposition~\ref{p:Creg} can be generalized in two different directions.
On the one hand, one can consider the class $\mathcal C$ of groups 
for which $H^2_b(\Gamma;\pi)\neq 0$ for some mixing unitary representation
$\pi$ of $\Gamma$, which is not necessarily the regular representation 
$\ell^2(\Gamma)$. The class $\mathcal C$ was also introduced by
Monod and Shalom~\cite{MSannals}, and their results used above for $\Creg$
apply more generally to $\mathcal C$. It is at present unknown whether
the inclusion $\Creg\subset\mathcal C$ is strict.
On the other hand, Thom~\cite{Thom} has introduced the following variant of $\Creg$.
\begin{defn}[\cite{Thom}]
  A countable group $\Gamma$ is in $\Dreg$ if
  $\dim_{L\Gamma}QH^1(\Gamma;\ell^2(\Gamma))\neq 0$, where $QH^1$
  denotes the first quasi-cohomology and $L\Gamma$ is the group von
  Neumann algebra of~$\Gamma$.
\end{defn}
It is as yet unknown whether $\Creg$ and $\Dreg$ agree. As far as presentability by products goes, both are equally
good:
\begin{prop}
Groups in $\Dreg$ are not presentable by products.
\end{prop}
\begin{proof}
  A standard exact sequence argument shows that if $\Gamma$ is in
  $\Dreg$ then $\ltb 1 \Gamma \neq 0$ or $\Gamma$ is in
  $\Creg$~\cite[Lemma~2.8]{Thom}. In the first case the
  conclusion follows from Proposition~\ref{firstl2prop}, in the second
  case it follows from Proposition~\ref{p:Creg}.
\end{proof}

\section{Thompson's groups}\label{s:thompson}

Richard Thompson's groups $F$, $T$ and $V$ are interesting test cases for many issues 
in group theory. We refer to the survey by Cannon et.~al.~\cite{CFP} for their basic properties.

The groups $T$ and $V$ are simple, and are therefore trivially not presentable by products.
For $F$ we have:
\begin{prop}\label{p:thompson}
The Thompson group $F$ is not presentable by products.
\end{prop}
\begin{proof}
Suppose $\Gamma_1$, $\Gamma_2\subset F$ are commuting infinite subgroups such that the 
multiplication map $\Gamma_1\times\Gamma_2\longrightarrow F$ is surjective onto a finite 
index subgroup $\Gamma\subset F$. There is a normal finite index subgroup $\bar\Gamma\subset F$
contained in $\Gamma$. Since $\bar\Gamma$ is normal in $F$, it contains the commutator 
subgroup $[F,F]$ by~\cite[Theorem~4.3]{CFP}. Since $[F,F]$ is normal in $F$, it is also
normal in $\Gamma$. The quotient $\Gamma/[F,F]$ is Abelian.

Since $\Gamma_i$ and $[F,F]$ are both normal in $\Gamma$, their intersection 
$\bar\Gamma_i = \Gamma_i\cap [F,F]$ is normal in $[F,F]$.
However, $[F,F]$ is a simple group~\cite[Theorem~4.5]{CFP}. Thus $\bar\Gamma_i$ is 
trivial or all of $[F,F]$. If $\bar\Gamma_i$ is trivial, then the composition
$$
\Gamma_i\hookrightarrow \Gamma\longrightarrow\Gamma/[F,F]
$$
is injective, and so $\Gamma_i$ must be Abelian. But then $\Gamma_i$ is an infinite central subgroup of $\Gamma$. This contradicts the fact that every finite index subgroup 
of $F$ must have trivial centre. For $F$ itself this is proved in~\cite[p.~229]{CFP}, and that 
proof applies to all finite index subgroups.

The only possibility left is that both $\bar\Gamma_i$ equal $[F,F]$. But then $[F,F]$ is 
contained in $\Gamma_1\cap\Gamma_2$, and so must be Abelian by the discussion in 
Section~\ref{s:prelim}. This contradicts the fact that $[F,F]$ is an infinite simple group
by~\cite[Theorem~4.5]{CFP}. This contradiction proves that $F$ can not be presentable 
by a product.
\end{proof}
This proposition can not be proved using the rank gradient, cost or the first $L^2$-Betti number, since $F$
contains copies of itself with positive index $>1$, which immediately implies the vanishing of its 
rank gradient, and the vanishing of $\CC(F)-1$ and of $\ltb 1 {F}$. In spite of various recent claims,
at the time of writing it seems to be still unknown whether $F$ is amenable.
If this were true, it would imply that the bounded cohomology of $F$ is trivial, and that 
$F$ is not $C^*$-simple, in particular $F$ would not be a Powers group. Note however that $F$ is not 
elementary amenable~\cite[Theorem~4.10]{CFP}, and so this is certainly a very different example 
from the elementary amenable groups discussed in Example~\ref{ex:amenable} and Corollary~\ref{c:Sol}.

\section{Automorphism groups of free and free Abelian groups}\label{s:aut}

In this section we test the obstructions against presentability by products in
the examples of automorphism groups of free Abelian as well as non-Abelian free groups.
In both cases we prove that the groups in question are not presentable by products.

\subsection{Automorphism groups of free Abelian groups}

The questions of presentability by products for $\GL(n,\Z)$ and for
$\SL(n,\Z)$ are equivalent, since the latter is a finite index
subgroup in the former. These groups are residually finite with
vanishing rank gradient~\cite{lack}, as shown by the consideration of
congruence subgroups.  Thus the obstructions of Section~\ref{s:cost}
do not apply.  Moreover, for~$n \in \N_{\geq 3}$, the
groups~$\GL(n,\Z)$ and $\SL(n,\Z)$ are not in the
class~$\Creg$~\cite[Theorem~1.4]{MSjdg}. These groups are not
$C^*$-simple, since they have non-trivial centres. However, it is
known that $\PSL(n,\Z)$ is $C^*$-simple by a result of Bekka, Cowling
and de~la~Harpe~\cite{BCH}. Whether $\PSL(n,\Z)$ is a Powers group for
$n\geq 3$ is unknown; compare~\cite{BridsonHarpe,Harpe2}.  Thus none
of the high-tech obstructions can be used to prove that for any $n\geq 2$ the 
groups $\SL(n,\Z)$ are not presentable by products.
Nevertheless, this is true, as it is a special case of the following:
\begin{prop}\label{p:SL}
Suppose $G$ is a connected semisimple Lie group with finite centre and rank $\geq 2$.
If $\Gamma\subset G$ is an irreducible lattice, then $\Gamma$ is not presentable by products.
\end{prop}
\begin{proof}
  Assume for a contradiction that $\Gamma_1$, $\Gamma_2\subset\Gamma$
  are infinite commuting subgroups such that the multiplication map
  $\Gamma_1\times\Gamma_2\longrightarrow\Gamma$ is surjective onto a
  finite index subgroup $\bar\Gamma\subset\Gamma$. Then $\bar\Gamma$
  is also an irreducible lattice. Since the $\Gamma_i$ are infinite
  normal subgroups in $\bar\Gamma$, the Margulis normal subgroup
  theorem~\cite[Chapter~IV]{margulis}, \cite[Chapter~8]{zimmer} implies
  that they have finite index in $\bar\Gamma$. Thus their
  intersection, which is a central subgroup, also has finite index,
  and so $\Gamma$ is virtually Abelian. This is absurd.
\end{proof}

Of course, for the case of $\GL(n,\Z)$ there is also an elementary
argument. One can find two elements in $\GL(n,\Z)$
that are diagonalizable over $\C$ and (whose non-trivial powers) have
no non-trivial common proper invariant subspace in~$\C^n$. Hence, the elements of a
finite index subgroup of~$\GL(n,\Z)$ can not have a common invariant
subspace in~$\C^n$. Assume $\GL(n,\Z)$ were presentable by a product
of subgroups~$\Gamma_1$ and~$\Gamma_2$. Using the fact that $\Gamma_1$
and $\Gamma_2$ commute, one could find a non-zero subspace~$E \subset
\C^n$ on which all elements of one of the factors, say~$\Gamma_2$, act
as multiples of the identity, and such that this subspace would
also be~$\Gamma_1$-invariant. Thus $E$ would be $(\Gamma_1 \cup
\Gamma_2)$-invariant. It would follow by what we said at the beginning that
$E=\C^n$, contradicting the assumption that $\Gamma_2$ is infinite.

\subsection{Automorphism groups of non-Abelian free groups}

Let $F_n$ be a free group on $n>1$ generators, $\Aut(F_n)$ its
automorphism group, and $\Out(F_n)=\Aut(F_n)/\Inn(F_n)$ its group of
outer automorphisms.
We use the following terminology:
\begin{defn}
  An element in~$\Out(F_n)$ is called \emph{reducible} if it leaves
  invariant the conjugacy class of a free factor in $F_n$, and it is
  called \emph{irreducible} otherwise.

  An element $g\in \Out(F_n)$ is called {\it fully irreducible} if
  $g^k$ is irreducible for all~$k\neq 0$.
\end{defn}

  Fully irreducible elements are sometimes called irreducible with
  irreducible powers (iwip), cf.~\cite{L}.
In $\Out(F_n)$ these elements play a r\^ole analogous to that of 
pseudo-Anosov elements in mapping class groups.

We now prove:
\begin{prop}\label{p:free}
 If $n \in \bN_{>1}$, then 
 the groups $\Aut(F_n)$ and $\Out(F_n)$ are
  not presentable by products.
\end{prop}
\begin{proof}
  We begin with the case of~$\Out(F_n)$. For $n=2$ this reduces to 
  $\GL(2,\Z)$, so there is nothing to prove. For $n\geq 3$ we may 
  appeal to Proposition~\ref{p:Powers}, since~$\Out(F_n)$ is a
  Powers group by a result of Bridson and de~la~Harpe~\cite[Theorem~2.6]{BridsonHarpe}.
  
  Instead of using the Powers property, we can give a direct proof by
  contradiction. It follows from a result of Baumslag and
  Taylor~\cite[Prop.~1]{BT} that $\Out(F_n)$ is virtually
  torsion-free. Thus, by the discussion in Section~\ref{s:prelim}, we
  may assume that we have a torsion-free finite index subgroup
  $\Gamma\subset \Out(F_n)$ together with two non-trivial commuting
  subgroups $\Gamma_1, \Gamma_2\subset\Gamma$ such that the
  multiplication
  homomorphism~$\Gamma_1\times\Gamma_2\longrightarrow\Gamma$ is
  surjective.

  Since $\Gamma$ has finite index in~$\Out(F_n)$, there exists a fully
  irreducible element~$g \in \Gamma$. By a result of Lustig~\cite{L} the 
  centraliser~$C_\Gamma(g)$ of~$g$ is virtually cyclic. Related 
  statements appear in the work of Bestvina, Feighn and Handel on the 
  Tits alternative for $\Out(F_n)$; see e.g.~\cite[Theorem~2.14]{BFH}.
  
  We can write $g = g_1 \cdot g_2$ with certain~$g_1 \in \Gamma_1$ and
  $g_2 \in \Gamma_2$. 
As $g$ is non-trivial, we may assume that so is $g_1$; note
  that~$g_1 \in C_\Gamma(g)$. Moreover, there exists an element~$g_2'
  \in \Gamma_2 \setminus \{1\}$ with~$g_2' \in C_\Gamma(g)$. If
  $g_2 \neq 1$ then we can take~$g_2' = g_2$, and if $g_2 = 1$ we may
  choose any non-trivial element of~$\Gamma_2$ for~$g_2'$. As both
  $g_1$ and $g_2'$ have infinite order and are contained in the virtually
  cyclic group $C_\Gamma(g)$, they have common non-trivial powers.
  This shows that $\Gamma_1\cap\Gamma_2$ is infinite, and so the 
  centre of $\Gamma$ is infinite by Lemma~\ref{l:center}. 

This is a contradiction, since $\Gamma$ must in fact have trivial centre;
compare~\cite{BFH}. (One way to see this is to check that $\Gamma$
contains two fully irreducible elements with distinct stable and unstable 
laminations.) This completes the direct proof that $\Out(F_n)$ is not presentable 
by products.

Next consider the extension
\begin{equation}\label{ext}
  1\longrightarrow F_n\longrightarrow \Aut(F_n)\stackrel{\pi}{\longrightarrow} \Out(F_n)\longrightarrow 1 \ .
\end{equation}
We may pull back this extension to a torsion-free finite index
subgroup of $\Out(F_n)$, so that the assumption on the quotient
in~\cite[Prop.~3.9]{KL} is satisfied by what we just
proved. Now~\cite[Prop.~3.9]{KL} tells us that $\Aut(F_n)$ is not
presentable by a product since the extension~\eqref{ext} does not
split when restricted to any finite index subgroups.  This completes
the proof of Proposition~\ref{p:free}.
\end{proof}

\begin{rem}
The direct argument for $\Out(F_n)$ could be rephrased to argue that the infinite 
cyclic subgroups generated by certain fully irreducible elements are acentral.
\end{rem}

\begin{rem}\label{rem:Bestvina}
After we proved directly that $\Out(F_n)$ is not presentable by products, we tried 
to find out whether $\Out(F_n)$ is (known to be) in~$\Creg$. In reply to our question,
Bestvina and Fujiwara told us that a proof of this statement will be contained in a 
forthcoming paper~\cite{BBF}. Since then, Hamenst\"adt~\cite{Ham3} has given such
a proof.
\end{rem}

The rank gradient, the cost, or the first $L^2$-Betti number cannot be used to prove
Proposition~\ref{p:free} in view of our next result:
\begin{prop}\label{p:Outcheap}
  Let $n \in \bN_{\geq 3}$. The groups $\Aut(F_n)$ and $\Out(F_n)$ are
  cheap. Their first $L^2$-Betti numbers and their rank gradients vanish.
\end{prop}
\begin{proof}
The groups in question are residually finite. For $\Aut(F_n)$ this is a classical result of 
Baumslag, whereas for $\Out(F_n)$ it was proved by Grossman~\cite{gross}. Thus, by the 
discussion in Subsection~\ref{ss:comparison}, we only have to prove the vanishing of the rank gradient.
For this we use again the result of Ab\'ert and Nikolov~\cite[Prop.~13]{AN} about extensions.
For $\Aut(F_n)$ we apply the result to the extension~\eqref{ext}. The group on the left is finitely generated and the 
group on the right admits finite quotients of arbitrarily large order.

Similarly for $\Out(F_n)$ we consider the extension
$$
  1\longrightarrow \IA_n\longrightarrow \Out(F_n)\longrightarrow \GL(n;\Z)\longrightarrow 1 \ .
$$
The groups on the left and on the right are infinite, and the kernel
$\IA_n$ is finitely generated by a classical result of
Magnus. Again the group on the right has finite quotients of arbitrarily large order.
\end{proof}
\begin{rem}
The argument for $\Aut(F_n)$ also works for $n=2$. The argument for $\Out(F_n)$
however breaks down for $n=2$ since $\IA_2$ is trivial. In this case $\Out(F_2)=\GL(2;\Z)$
has positive rank gradient as it is virtually free~\cite{lack}.
\end{rem}

\section{Ends, free products, and connected sums}\label{s:free}

In this section we consider free products of groups, and, more
generally, groups with infinitely many ends.

\begin{prop}\label{p:ends}
Groups with infinitely many ends are not presentable by products.
\end{prop}
\begin{proof}
  It is well known that groups with infinitely many ends have positive
  first $L^2$-Betti number; see for example~\cite[Chapter~4]{ABCKT}
  or~\cite[Cor.~1]{BV}. Therefore the result follows from
  Proposition~\ref{firstl2prop}. 

  Alternatively we could use the fact that groups with infinitely many
  ends are in~$\Creg$, as proved by Monod and
  Shalom~\cite[Corollary~7.9]{MSjdg}, and appeal to
  Proposition~\ref{p:Creg}. Notice however, that in contrast to the
  result about the first $L^2$-Betti number, the proof of Monod and
  Shalom uses Stallings's structure theorem for groups with
  infinitely many ends.
  
  Finally, a completely elementary argument is possible as well.
  Freudenthal and Hopf proved that a group with
  infinitely many ends cannot be a direct product of infinite groups.
  The argument given by Freudenthal~\cite[7.10]{freude} in fact
  proves the more general statement of this theorem. For the 
  convenience of the reader we repeat this argument briefly.
  
  Let $\Gamma$ be a group with infinitely many ends, and $\Gamma_1$, 
  $\Gamma_2$ commuting infinite subgroups for which the multiplication map
  $\varphi\colon\Gamma_1\times\Gamma_2\longrightarrow\Gamma$ is surjective onto a 
  finite index subgroup. Since the number of ends is unchanged by passage to a 
  finite index subgroup, we may assume that $\varphi$ is surjective.
  The assumption that $\Gamma$ has more than one end implies that 
  there is an element $g\in\Gamma$ of infinite order for which $g^n$
  and $g^{-n}$ belong to two different ends $\mathfrak{e}$ and $\mathfrak{e}'$
  as $n\rightarrow\infty$; see~\cite[7.6]{freude}.
  
  Under the action of $\Gamma$ on its space of ends, the infinite
  cyclic subgroup $T$ generated by $g$ fixes $\mathfrak{e}$ and
  $\mathfrak{e}'$. Write $g=g_1g_2$ with $g_i\in\Gamma_i$. The $g_i$
  commute with $T$, and so both $g_i$ also fix $\mathfrak{e}$ and
  $\mathfrak{e}'$.  Since $g$ has infinite order, we may assume that
  so does $g_1$. Then $g_1$ generates an infinite cyclic subgroup $T'$
  of $\Gamma_1$ that fixes $\mathfrak{e}$ and $\mathfrak{e}'$.  As
  $\Gamma_2$ commutes with $T'$, it contains a subgroup $\Gamma_2'$ of
  index at most $2$ that also fixes $\mathfrak{e}$ and
  $\mathfrak{e}'$~\cite[7.7]{freude}.  As $\Gamma_1$ commutes with
  $\Gamma_2'$, it contains a subgroup $\Gamma_1'$ of index at most $2$
  which also fixes $\mathfrak{e}$ and
  $\mathfrak{e}'$~\cite[7.7]{freude}. Thus $\Gamma$ has a subgroup of
  index at most $4$ which fixes $\mathfrak{e}$ and
  $\mathfrak{e}'$. This contradicts the assumption that $\Gamma$ has
  infinitely many ends.
\end{proof}

\begin{cor}
\label{freeprodcoro}
  Let $\Delta_1$ and $\Delta_2$ be two non-trivial groups. Then
  the free product~$\Delta_1 \star \Delta_2$ is presentable by a product if
  and only if~$\Delta_1 \cong \Z/2 \cong \Delta_2$.
\end{cor}
\begin{proof}
  On the one hand, $\Z/2\star\Z/2$ is virtually infinite cyclic, and
  therefore presentable by a product.  On the other hand, if one of
  the groups has order at least~$3$, then their free product has
  infinitely many ends whence Proposition~\ref{p:ends} applies.
  Alternatively we can use Example~\ref{ex:free} to see that there are
  infinite acentral subgroups of infinite index and apply Proposition~\ref{prop:acentral},
  or we can use Proposition~\ref{p:Powers} in conjunction
  with the fact that these free products are Powers groups; see 
  Bridson and de~la~Harpe~\cite[Theorem~2.2]{BridsonHarpe}.
\end{proof}

\begin{rem}
  Lackenby~\cite[Prop.~3.2]{lack} proved that the rank gradient of a
  free product $\Delta_1 \star \Delta_2$ of non-trivial groups is positive
  if at least one of the free factors has order~$>2$. Therefore, for
  residually finite groups Corollary~\ref{freeprodcoro} also follows from
  Proposition~\ref{p:RG}.
\end{rem}

We can use the last corollary to put restrictions on the connected sum
decompositions of manifolds dominated by products. Suppose $N=N_1\#
N_2$ is a connected sum of two closed oriented $n$-manifolds, and
$P=X_1\times X_2$ is a non-trivial product of closed oriented manifolds
with~$P\geq N$. Then, collapsing one or the other summand of~$N$ to a
point, we see that $P\geq N_1$ and $P \geq N_2$. Thus, for $N$ to be
dominated by a product it is necessary that its connected
summands~$N_i$ are also dominated by products. However, this
necessary condition is not sufficient.

\begin{thm}\label{t:sums}
  If $N$ is a closed, oriented, connected rationally essential
  manifold that is dominated by a non-trivial product $P\geq N$ and that admits 
  a connected sum decomposition $N=N_1\# N_2$, then one of the summands is simply 
  connected, and the fundamental group of the other summand is presentable by a product.
\end{thm}
\begin{proof}
Clearly we may assume that $N$ has dimension $\geq 3$. Then its fundamental group is the free
product of the fundamental groups of the $N_i$, and, since $N$ is assumed rationally essential, at least one
of these free factors must be infinite. If the other free factor is non-trivial, Corollary~\ref{freeprodcoro} tells us that 
$\pi_1(N)$ is not presentable by a product, which contradicts~\cite[Theorem~1.4]{KL}. Thus one of the $N_i$ 
is simply connected, the other one is rationally essential, and its fundamental group is presentable by a product
by~\cite[Theorem~1.4]{KL}. 
\end{proof}

\begin{ex}
  In every dimension $n\geq 2$, the $n$-torus $T^n$ is a product, but
  $T^n\# T^n$ is not dominated by a product.
\end{ex}

\begin{rem}
  Notice however, that not all non-trivial connected sums are not
  dominated by a product; for instance, $\C P^2 \# \C P^2$ is
  dominated by a product~\cite[Proposition~7.1]{KL}. 
\end{rem}

\section{Final remarks}\label{s:final}

\subsection{Extension to subnormal subgroups}\label{ss:normal}

In this paper we have proved that various groups are not presentable by products. By definition, this notion refers to all
subgroups of finite index, in particular the finite index normal subgroups. It turns out that in many
cases one can treat all infinite normal subgroups of our groups, regardless of whether they have finite 
index, or not. This leads to the following result:

\begin{thm}\label{t:subnormal}
Let $\Gamma$ be a group from the following list of examples:
\begin{enumerate}[\normalfont{(MCG)}]
\item[\normalfont{(H)}] hyperbolic groups that are not virtually
  cyclic,
\item[\normalfont{(N-P)}] fundamental groups of closed Riemannian
  manifolds of non-positive sectional curvature of rank one and of
  dimension~$\geq 2$,
\item[\normalfont{(LAT)}] irreducible lattices in connected semisimple Lie groups with finite centre 
and rank~$\geq 2$,
\item[\normalfont{(MCG)}] mapping class groups of closed oriented
  surfaces of genus $\geq 1$,
\item[\normalfont{(OUT)}] outer automorphism groups of free groups of rank $\geq 2$,
\item[\normalfont{(END)}] groups with infinitely many ends.
\end{enumerate}
Then no infinite subnormal subgroup of $\Gamma$ is presentable by a product.
\end{thm}

Recall that a subgroup $\Gamma_0\subset \Gamma$ is \emph{subnormal} if there is a descending sequence of subgroups
$\Gamma_0\subset\Gamma_1\subset\ldots\subset\Gamma_k\subset\Gamma_{k+1}=\Gamma$ such that
$\Gamma_i$ is normal in $\Gamma_{i+1}$ for all $i \in\{0, \dots, k\}$.

In order to give a quick and uniform proof for almost all the different cases we use the fact that all the groups 
in the theorem, except the lattices in (LAT), are in $\Creg$; compare the survey table in the Appendix. It was proved by Monod and 
Shalom~\cite[Proposition~7.4]{MSannals} that if $\Gamma$ is in $\Creg$, then so is every infinite normal
subgroup. Theorem~\ref{t:subnormal} then follows from Proposition~\ref{p:Creg} by induction on the length of 
the chain of subnormal subgroups. In the case of the lattices in (LAT), the Margulis normal subgroup theorem~\cite[Chapter~IV]{margulis},
\cite[Chapter~8]{zimmer} implies that every infinite subnormal subgroup has finite index. The conclusion then follows from Proposition~\ref{p:SL}.

The case (N-P) in the Theorem can be generalized further by considering \Catz-groups in the sense of~\cite{bh}. 
Let $\Gamma$ be any discrete group that admits a proper, 
minimal, isometric action without fixed points at infinity on a proper,
  irreducible \Catz-space~$X$ with finite-dimensional boundary. If $X$
  is not the real line, then no infinite subnormal subgroup of $\Gamma$ is presentable by products. This is 
  implicit in a result of Caprace and Monod~\cite[Theorem~1.10]{cmstructure}.
  
To put this extension into context, recall that an action of a group on a \Catz-space is \emph{minimal}
if this space does not contain a non-empty invariant closed
convex (proper) subspace. As in the Riemannian case, a \Catz-space is
\emph{irreducible} if it does not admit a non-trivial isometric
splitting as a direct product.
If a discrete group acts cocompactly via isometries on a proper
\Catz-space~$X$, then the boundary of~$X$ is automatically
finite-dimensional~\cite[Theorem~C]{kleiner}. Moreover, if a discrete
group acts properly discontinuously, minimally, and cocompactly via
isometries on a \Catz-space without Euclidean factors, then this
action does not have any fixed points at
infinity~\cite[Corollary~2.7]{ab}.

\subsection{Relations with geometric and with measurable group theory}

The property of being or not being presentable by a product is not
always shared by groups that are equivalent under one of the usual
equivalence relations considered in geometric group theory.

\begin{thm}\label{t:non-invariance}
  The property of being presentable by products is not invariant under
  quasi-isometries, under measure equivalence, or under orbit
  equivalence.
\end{thm}
\begin{proof}
  The isometry group of the polydisk~$\Hyp^2 \times \Hyp^2$ contains
  both reducible and irreducible cocompact lattices. The reducible
  ones are trivially presentable by products, whereas the irreducible
  ones are not presentable by
  products~\cite[Corollary~4.2]{KL}. However, all these lattices are
  quasi-isometric to each other by the Milnor--\v Svarc lemma.  This
  shows that presentability by products is not a quasi-isometry
  invariant property.

  All infinite amenable groups admit orbit equivalent measure
  preserving free actions on standard Borel probability
  spaces~\cite{ornsteinweiss}. Obviously, there are many amenable
  groups that are presentable by a product, for instance free Abelian
  groups of non-zero rank. However, there are also amenable groups
  that are not presentable by products; see
  Example~\ref{ex:amenable}. Thus, presentability by products is not
  invariant under orbit equivalence.

  The examples mentioned in the previous paragraph also show that
  presentability by products is not invariant under measure
  equivalences (the class of groups that are measure equivalent
  to~$\Z$ equals the class of all infinite countable amenable
  groups~\cite{furman}).
\end{proof}

In spite of Theorem~\ref{t:non-invariance}, many of the obstructions
against presentability by products that we have discussed in this
paper have strong invariance properties under these equivalence
relations.  For example, the non-vanishing of the first $L^2$-Betti
number is a quasi-isometry invariant~\cite[p.~19 and~224]{gromovasym},
\cite[p.~314]{BV}. 
Moreover, Gaboriau proved that the
vanishing of the first $L^2$-Betti number is an orbit equivalence
invariant~\cite[Th\'eor\`eme~3.12]{gaboriauIHES} 
and a measure equivalence
invariant~\cite[Th\'eor\`eme~6.3]{gaboriauIHES}. 
Next, being expensive is an orbit
equivalence invariant and a measure equivalence invariant for groups
with fixed price as the cost of a group is defined in terms of its
orbit relations~\cite[Proposition~VI.5,~VI.6]{gaboriauInvent}.
Finally, the non-vanishing of the
second bounded cohomology with coefficients in the regular
representation is a measure equivalence
invariant~\cite[Corollary~7.6]{MSannals}.  
Whether it is invariant under quasi-isometries seems to be unknown; see Monod's 2006 ICM talk~\cite[Problem~J]{monodICM}.
 

\section*{Appendix: Overview of results}

Table~\ref{overviewtable} surveys the applicability of different
criteria to proving that certain classes of groups are not presentable
by products. The first column lists certain test classes of groups; of
the other columns each corresponds to a way of concluding that groups
are not presentable by products. The ``ad hoc'' column refers to the
direct, low-tech, hands-on argument relying on information about the
sizes of centralisers, including in particular the arguments about 
acentral extensions of Section~\ref{s:acentral}. The other columns each use some high-brow theory.

\newlength{\saveparindent}
\setlength{\saveparindent}{\parindent}

\begin{table}
\begin{minipage}{\linewidth}
  \newcommand{\tfootnotemark}[1]{\,{}$^{#1}$}
  \newcommand{\tfootnotetext}[2]{\makebox[\the\saveparindent][r]{{}$^{#1}$\,} & #2\\}
  \begin{center}
  {\footnotesize
  \setlength{\extrarowheight}{10pt}
  \begin{tabular}{l|ccccc} 
    properties of~$\Gamma$ & ad hoc & $\ltb 1 \Gamma > 0$ 
  & $\CC(\Gamma) > 1$ & $\Gamma\in\mathcal{P}$ & $\Gamma\in\Creg$ \\ \hline
    {hyperbolic (non-el.)} 
  & \cite[Prop.~3.6]{KL} 
  & --- & ---\tfootnotemark{1} 
  & ---\tfootnotemark{8} & \cite[Thm.~3]{MMS}  
  \\ 
    {MCG in genus $\geq 3$}
  & \cite[Prop.~3.8]{KL} 
  & --- & ---\tfootnotemark{2} 
  & \cite[Thm.~2.2]{BridsonHarpe} & \cite[Thm.~4.5]{Ham}  
  \\ 
    {$\Out(F_n)$ for $n\geq 3$} 
  & Prop.~\ref{p:free} 
  & --- & ---\tfootnotemark{3} 
  & \cite[Thm.~2.6]{BridsonHarpe} & \cite[Cor.]{Ham3} 
  \\ 
    {$\vert e(\Gamma)\vert=\infty$}
  & Prop.~\ref{p:ends}
  & $\checkmark$\tfootnotemark{0} & $\checkmark$ 
  & ---\tfootnotemark{8} & \cite[Cor.~7.9]{MSjdg} 
  \\ 
    {$\Delta_1\star\Delta_2$, $\vert\Delta_i\vert>i$ } 
  & Example~\ref{ex:free} 
  & $\checkmark$\tfootnotemark{0} & $\checkmark$ 
  & \cite[Prop.~8]{Harpe1} & \cite[Cor.~7.9]{MSjdg} 
  \\ 
    {(N-P), rank~$1$ } 
  & \cite[Prop.~3.7]{KL} 
  &  --- & ---\tfootnotemark{1} 
  & ? & \cite[Thm.~2]{Ham2} 
  \\ 
    {(N-P), irred., rk $\geq 2$} 
  & \cite[Thm.~4.1]{KL}
  & --- & ---\tfootnotemark{4} &  ? & ---\tfootnotemark{5}  
   \\ 
    {$\GL(n,\Z)$, $n\geq 3$} 
  & Prop.~\ref{p:SL}
  & --- & ---\tfootnotemark{4} &  ---\tfootnotemark{11} & ---\tfootnotemark{5}  
 \\ 
     {Thompson $F$} 
  & Prop.~\ref{p:thompson}
  & --- & ---\tfootnotemark{10} & ? & ?  
  \\ 
     {$\Gamma\subset \Sol^3$ a lattice} 
  & Cor.~\ref{c:Sol}
  & --- & ---\tfootnotemark{6} & ---\tfootnotemark{9} & ---\tfootnotemark{7}  
  \\ 
\end{tabular}
  }
  \end{center}

  \medskip

  \noindent 
  \rule{5em}{.4pt}

  \medskip

  \noindent
  {\footnotesize
  \begin{tabularx}{\linewidth}{r@{}X}
    \tfootnotetext{0}{well known}
    \tfootnotetext{1}{See Example~\ref{ex:hyper}.}  
    \tfootnotetext{2}{See Proposition~\ref{p:MCG}.}
    \tfootnotetext{3}{See Proposition~\ref{p:Outcheap}.}  
    \tfootnotetext{4}{Lattices in higher rank Lie groups are
      cheap~\cite[Cor.VI.30]{gaboriauInvent}.}   
    \tfootnotetext{5}{Lattices in almost simple higher rank Lie groups are
      not in $\Creg$~\cite[Thm.~1.4]{MSjdg}.}
    \tfootnotetext{6}{See Lemma~\ref{l:lack} or~\cite[Prop.~35.1 (i)]{KM}.}
    \tfootnotetext{7}{The bounded cohomology of amenable groups vanishes.}
      \tfootnotetext{8}{See Example~\ref{ex:SL2}.}  
   \tfootnotetext{9}{A non-trivial amenable group cannot be $C^*$-simple; cf.~\cite{Harpe2}.}
   \tfootnotetext{10}{$F$ contains itself with finite index $>1$.}
  \tfootnotetext{11}{$\GL(n,\Z)$ has non-trivial center, and so cannot be $C^*$-simple.}
\end{tabularx}}
\end{minipage}

\bigskip

\caption{Overview of results}\label{overviewtable}
\end{table}

The first test class of groups are the non-elementary hyperbolic
groups, denoted (H) in~\cite{KL}. Non-elementary is the same thing as
not virtually cyclic.

The second class are the mapping class groups of closed oriented
surfaces of genus at least~$3$. We omit genus $1$ and $2$
because they have special features that do not occur in high genus, e.g.,~they
have non-trivial centres.

The third class are the outer automorphism groups of free groups
$\Out(F_n)$, where we assume that $n\geq 3$. For $n=2$ one has
$\Out(F_2)=\GL(2;\Z)$.

The case of a free product $\Delta_1\star\Delta_2$ with
$\vert\Delta_i\vert>i$ is contained in the more general situation of a
group with infinitely many ends considered separately here. However,
not all criteria that apply to free products generalize to groups with
infinitely many ends.

By (N-P) we mean the class of fundamental groups of closed oriented
manifolds of non-positive curvature, as considered in~\cite{KL}.  The
results about this class can be extended to groups admitting suitable
actions on CAT(0)-spaces; see Subsection~\ref{ss:normal}.

For the Thompson group $F$ amenability seems to be an open question,
but it is certainly not elementary amenable.
The final example concerns the fundamental groups of $3$-manifolds
carrying the Thurston geometry $\Sol^3$. These are elementary 
amenable, and they show that none of the
high-tech obstructions against presentability by products apply to 
arbitrary acentral extensions.

Where a criterion does work for a class of groups, the corresponding
entry in the table gives the earliest reference for a complete proof
known to us. A horizontal line indicates that the criterion is not
applicable; this is explained in the footnotes.

The checkmarks in the cost column come from Theorem~\ref{t:gab}. If a
group can be shown not be presentable by products using the first
$L^2$-Betti number, then one can also use the cost for this
purpose. Conversely, if a group is cheap, then its first $L^2$-Betti
number vanishes, and this explains the horizontal lines without
footnotes in the $L^2$-Betti number column.


\bibliographystyle{amsplain}

\bigskip

\end{document}